\documentclass[11pt]{article}

\usepackage{amsthm,amsmath,amssymb,amsfonts,mdframed,graphicx,euscript}
\usepackage{color, colortbl}
\usepackage{xcolor}
\usepackage{array}
\usepackage{multirow}

\newtheorem{definition}{Definition}
\newtheorem{theorem}{Theorem}
\newtheorem{lemma}[theorem]{Lemma}

\newtheorem{remark}[theorem]{Remark}
\newtheorem{proposition}[theorem]{Proposition}
\newtheorem{corollary}[theorem]{Corollary}

\usepackage{tikz}
\usetikzlibrary{arrows.meta}
\tikzset{>={Latex[width=0.8mm,length=0.8mm]}}
\usetikzlibrary{positioning}
\usetikzlibrary{arrows}
\usepackage{tikz,amsmath, amssymb,bm,color}

\newcommand{\pattern}[4]{										
	\raisebox{0.6ex}{
		\begin{tikzpicture}[scale=0.35, baseline=(current bounding box.center), #1]
		\foreach \x/\y in {#4}		\fill[gray!20] (\x,\y) rectangle +(1,1);
		\draw (0.01,0.01) grid (#2+0.99,#2+0.99);
		\foreach \x/\y in {#3}		\filldraw (\x,\y) circle (6pt);
		\end{tikzpicture}}
}

\DeclareMathOperator{\sgn}{sign}

\title{Singleton mesh patterns in multidimensional permutations}

\author{Sergey Avgustinovich\footnote{Sobolev Institute of Mathematics, Prospekt Akademika Koptyuga 4, Novosibirsk, 630090, Russia.
{\bf Email:} \{avgust,\ vpotapov,\ taa\}@math.nsc.ru.}\ , Sergey Kitaev\footnote{Department of Mathematics and Statistics, University of Strathclyde, 26 Richmond Street, Glasgow G1, 1XH, United Kingdom.
{\bf Email:} sergey.kitaev@strath.ac.uk.}\ , Jeffrey Liese\footnote{Department of Mathematics, California Polytechnic State University, San Luis Obispo, CA 93407, USA. {\bf Email:} jliese@calpoly.edu.}\ , \\ Vladimir Potapov\footnotemark[1]\ , and Anna Taranenko\footnotemark[1]}

\begin{document}

\maketitle

\begin{abstract}
This paper introduces the notion of mesh patterns in multidimensional permutations and initiates a systematic study of singleton mesh patterns (SMPs), which are multidimensional mesh patterns of length 1. A pattern is avoidable if there exist arbitrarily large permutations that do not contain it. As our main result, we give a complete characterization of avoidable SMPs using an invariant of a pattern that we call its rank. We show that determining avoidability for a $d$-dimensional SMP $P$ of cardinality $k$ is an $O(d\cdot k)$ problem, while determining rank of $P$ is an NP-complete problem.  Additionally, using the notion of a minus-antipodal pattern, we characterize SMPs which occur at most once in any $d$-dimensional permutation.  Lastly, we provide a number of enumerative results regarding the distributions of certain general projective, plus-antipodal, minus-antipodal and hyperplane SMPs. \\


\end{abstract}

\section{Introduction}

Permutation patterns have attracted much attention in the literature in the last couple of decades \cite{Kit2011}.  The notion of a mesh pattern, generalizing several types of patterns, was introduced by Br\"and\'en and Claesson \cite{BrCl} to provide explicit expansions for certain permutation statistics as, possibly infinite, linear combinations of (classical) permutation patterns. Systematic studies of avoidance of mesh patterns of short length were conducted in \cite{Hilmarsson2015Wilf} and distribution of such patterns in \cite{KitZha}.

Singleton mesh patterns are a generalization of well-known permutation statistics including {\em left-to-right maxima}, {\em left-to-right minima}, {\em right-to-left maxima}, {\em right-to-left minima}, and others.  These patterns are a particular case of {\em quadrant marked mesh patterns} introduced in \cite{KitRem2012} and studied in several classes of permutations (e.g.\ \cite{KitRem2012-2,KitRemTie2012,QiuRem2017}). In particular, in \cite{KitRem2012-2},  classic enumeration results of Andr\'{e} \cite{Andre1,Andre2} on alternating permutations obtained in 1879 were refined by showing that the distribution of a certain quadrant marked mesh pattern is given by $(\sec(xt))^{1/x}$ on up-down permutations of even length and by $\int_0^t (\sec(xz))^{1+\frac{1}{x}}dz$ on down-up permutations of odd length.

The goal of this paper is to introduce the notion of a mesh pattern in multidimensional permutations and to initiate a systematic study of singleton multidimensional mesh patterns. We note that patterns in 3-dimensional permutations have been previously considered in the literature~\cite{AM2010,ZG2007}, as well as patterns in multidimensional objects \cite{KR2007}. However, the types of patterns introduced in this paper are new for dimensions higher than 2. Bringing the studies of (marked) mesh patterns, recorded in a long line of papers in the literature, to higher dimensions is a natural next-step in further developing the theory of permutation patterns.

A mesh pattern is avoidable if there exist arbitrarily large permutations avoiding it. The main result of this paper is Theorem~\ref{thm-main}, which gives a complete characterization of avoidable singleton mesh patterns in terms of their ranks. We show that finding the rank of a singleton mesh pattern is an NP-complete problem, while determining avoidability for a $d$-dimensional SMP $P$ of cardinality $k$ is an $O(d\cdot k)$ problem (see Corollary~\ref{complexity-avoidability}). Another interesting result is Theorem~\ref{one-occur-crit}, which characterizes singleton mesh patterns occurring at most once in any $d$-dimensional permutation using the notion of a minus-antipodal pattern.

The paper is organized as follows. In Section~\ref{prelim}, we introduce all necessary definitions and preliminary results. In Section~\ref{smp-characterization-sec}, we characterize avoidable multidimensional singleton mesh patterns. In Section~\ref{enumeration-sec}, we introduce four general classes of singleton mesh patterns (projective, hyperplane, plus-antipodal and minus-antipodal) and give a number of enumerative results for these patterns. In particular, we show how reduction in dimension can be used for projective and hyperplane patterns and we find the distributions of all 3-dimensional projective patterns. Also, in Section~\ref{enumeration-sec}, we find distribution of plus-antipodal patterns of next to maximum cardinality (see Theorem~\ref{plus-antipodal-thm}) and give asymptotics for the number of $d$-dimensional permutations with the maximum number of occurrences of a simplest non-empty plus-antipodal pattern (see Theorem~\ref{plus-antip-d-result}). In Section~\ref{generalization-any-mesh-sec}, we suggest generalizations of singleton mesh patterns on multidimensional permutations to mesh patterns of arbitrary length.  Studying these generalizations is largely outside of the goals of this paper, but we do provide a couple of relevant enumerative results and one bijective result. Finally, in Section~\ref{research-directions-sec}, we suggest a number of directions for further research.

\section{Preliminaries}\label{prelim}

 Let $\pi = \pi_1 \pi_2 \dots \pi_n$ be a permutation of length $n$ ($n$-permutation) in the symmetric group $S_n$.  As written, $\pi$ is in one-line notation, but it will often be useful for us to use two-line notation and we write \[\pi=\left(\begin{array}{cccc}1&2&\dots&n\\ \pi_1 & \pi_2 & \dots & \pi_n\\ \end{array}\right).\]
The {\em complement} of $\pi$, denoted by $c(\pi)$, is the permutation obtained from $\pi$ by replacing $\pi_i$ by $n+1-\pi_i$ for $i\in\{1,2,\ldots, n\}$. For example, if $\pi=2134$ then
$c(\pi)=3421$. The
{\em graph} of $\pi$, is the set of points $\{(i,\pi_i)\}_{i=1}^n$.  It is worth noting that these points are obtained from the columns of the two-line representation of $\pi$. The graph of the permutation can be visualized in the $xy$-plane and is usually called the {\em permutation diagram} of $\pi$.  The graph of
$\pi = 471569283$ is shown in Figure~\ref{fig:basic}. For any $n$-permutation $\pi$, we introduce $n$ new coordinate systems, each of which is centered at a point $(i,\pi_i)$.  We are interested in which quadrants (I, II, III or IV) other elements of $\pi$ are located in with respect to each coordinate system.  We use the standard ordering for our quadrants and this is also depicted in Figure~\ref{fig:basic}. \\[-3mm]

\begin{figure}
\begin{center}
\includegraphics[scale=0.7]{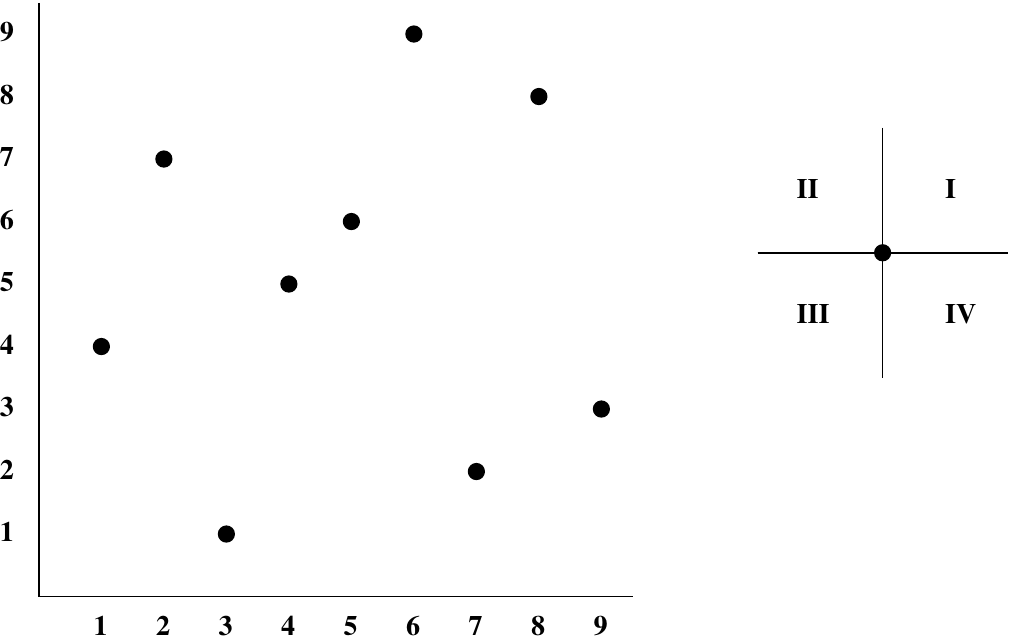}
\caption{The graph of $\pi = 471569283$}\label{fig:basic}
\end{center}
\end{figure}

\noindent
{\bf Singleton 2-dimensional mesh patterns.} We say that an element $\pi_i$ of $\pi$, represented by the point $(i,\pi_i)$, is an {\em occurrence} of the {\em singleton mesh pattern} $\pattern{scale=0.5}{1}{1/1}{1/1}$ (resp., $\pattern{scale=0.5}{1}{1/1}{0/1}$, $\pattern{scale=0.5}{1}{1/1}{0/0}$, $\pattern{scale=0.5}{1}{1/1}{1/0}$) if there are no points in quadrant I (resp., II, III, IV) in the coordinate system centered at $(i,\pi_i)$.  For example, the element 2 in the permutation in Figure~\ref{fig:basic} is an occurrence of the pattern $\pattern{scale=0.5}{1}{1/1}{1/0}$, but {\em not} of the pattern $\pattern{scale=0.5}{1}{1/1}{0/1}$ because there are five points in the forbidden area, in particular the elements 4, 5, 6, 7, and 9.  More generally, we can forbid elements from belonging to multiple quadrants.  For example, the permutation $\pi$ in Figure~\ref{fig:basic} has no occurrence of the pattern  $\pattern{scale=0.5}{1}{1/1}{0/1,1/0}$.  In this situation we say that $\pi$  {\em avoids} $\pattern{scale=0.5}{1}{1/1}{0/1,1/0}$.

A {\em left-to-right maximum} (resp., {\em minimum}) in a permutation $\pi$ is an element $\pi_i$ such that $\pi_i>\pi_j$ (resp., $\pi_i<\pi_j$) for $j<i$. A {\em right-to-left maximum} (resp., {\em minimum}) in a permutation $\pi$ is an element $\pi_i$ such that $\pi_i>\pi_j$ (resp., $\pi_i<\pi_j$) for $j>i$. Occurrences of $\pattern{scale=0.5}{1}{1/1}{1/1}$, $\pattern{scale=0.5}{1}{1/1}{0/1}$, $\pattern{scale=0.5}{1}{1/1}{0/0}$ and $\pattern{scale=0.5}{1}{1/1}{1/0}$ are precisely occurrences of right-to-left maxima, left-to-right maxima, left-to-right minima and right-to-left minima respectively.  Hence, the singleton mesh patterns generalize the notions of these permutation statistics. As each quadrant is either forbidden (shaded) or not, it is clear that the number of 2-dimensional singleton mesh patterns is $2^4=16$.\\[-3mm]

\noindent
{\bf $d$-dimensional permutations.}  A \textit{$d$-dimensional permutation $\Pi$ of length $n$} is an ordered $(d-1)$-tuple $(\pi^2,\pi^3, \dots , \pi^d)$ of $n$-permutations where for each $2\leq i\leq d$, $\pi^i=\pi_1^i\pi_2^i\dots\pi_n^i\in S_n$. For example, $(231,312,231)$ is a 4-dimensional permutation of length 3. We let $S^d_n$ denote the set of $d$-dimensional permutations of length $n$. Note that $S^2_n$ corresponds naturally to $S_n$.  We also generalize two-line notation to $d$-line notation and we write
\renewcommand{\arraystretch}{1.25}
\[\Pi=\left(\begin{array}{cccc}1&2&\dots&n\\
\pi^2_1 & \pi^2_2 & \dots & \pi^2_n\\
 \pi^3_1 & \pi^3_2 & \dots & \pi^3_n\\
  \vdots &  & \dots & \vdots\\
   \pi^d_1 & \pi^d_2 & \dots & \pi^d_n\\ \end{array}\right)=\left(\begin{array}{cccc}\pi^1_1&\pi^1_2&\dots&\pi^1_n\\
\pi^2_1 & \pi^2_2 & \dots & \pi^2_n\\
 \pi^3_1 & \pi^3_2 & \dots & \pi^3_n\\
  \vdots &  & \dots & \vdots\\
   \pi^d_1 & \pi^d_2 & \dots & \pi^d_n\\ \end{array}\right),\]
so that $\Pi$ corresponds naturally to a $d\times n$ matrix.  It is also helpful to let $\pi^1$ denote the permutation $12\dots n$ so that we can succinctly write \[\Pi=\left\{\pi^i_j\right\}_{\begin{subarray}{l} 1 \leq i\leq d \\ 1\le j\le n\end{subarray}}.\]  Motivated by two-line notation, we say that the columns of this matrix represent the \textit{elements of $\Pi$} which we denote by  $\Pi_i$.  In particular, we write $\Pi=\Pi_1\Pi_2\dots \Pi_n$ where $\Pi_i$ is the $d$-tuple $(i,\pi^2_i, \pi^3_i,\ldots, \pi^d_i)^T$.  Analogously, the \textit{graph} of a $d$-dimensional permutation $\Pi$ of length $n$ is the set of $d$-tuples $\{\Pi_i\}_{i=1}^n$. For example, if $\Pi = (\pi^2, \pi^3)$ is a $3$-dimensional permutation of length $5$ with $\pi^2 = 12534$ and $\pi^3 = 51243$, then we write
\renewcommand{\arraystretch}{1}
\[\Pi=\left(\begin{array}{c}\pi^1\\ \pi^2\\ \pi^3 \end{array}\right)=\left(\begin{array}{ccccc}
1&2&3&4&5\\
1&2&5&3&4\\
5&1&2&4&3\\
\end{array}\right),\]
or $\Pi=\Pi_1\Pi_2\Pi_3\Pi_4\Pi_5$ where
\begin{align*}\Pi_1&=(1,1,5)^T\\ \Pi_2&=(2,2,1)^T\\ \Pi_3&=(3,5,2)^T\\ \Pi_4&=(4,3,4)^T\\ \Pi_5&=(5,4,3)^T.\end{align*} The graph of $\Pi$ is the set  $\{\Pi_i\}_{i=1}^5$ and is depicted in Figure \ref{fig:3dperm}.  Note that the usual graphs of $\pi^2$ and $\pi^3$ can be seen as projections onto two of the coordinate planes in the graph of $\Pi$.

\begin{figure}
\begin{center}
\includegraphics[scale=0.5]{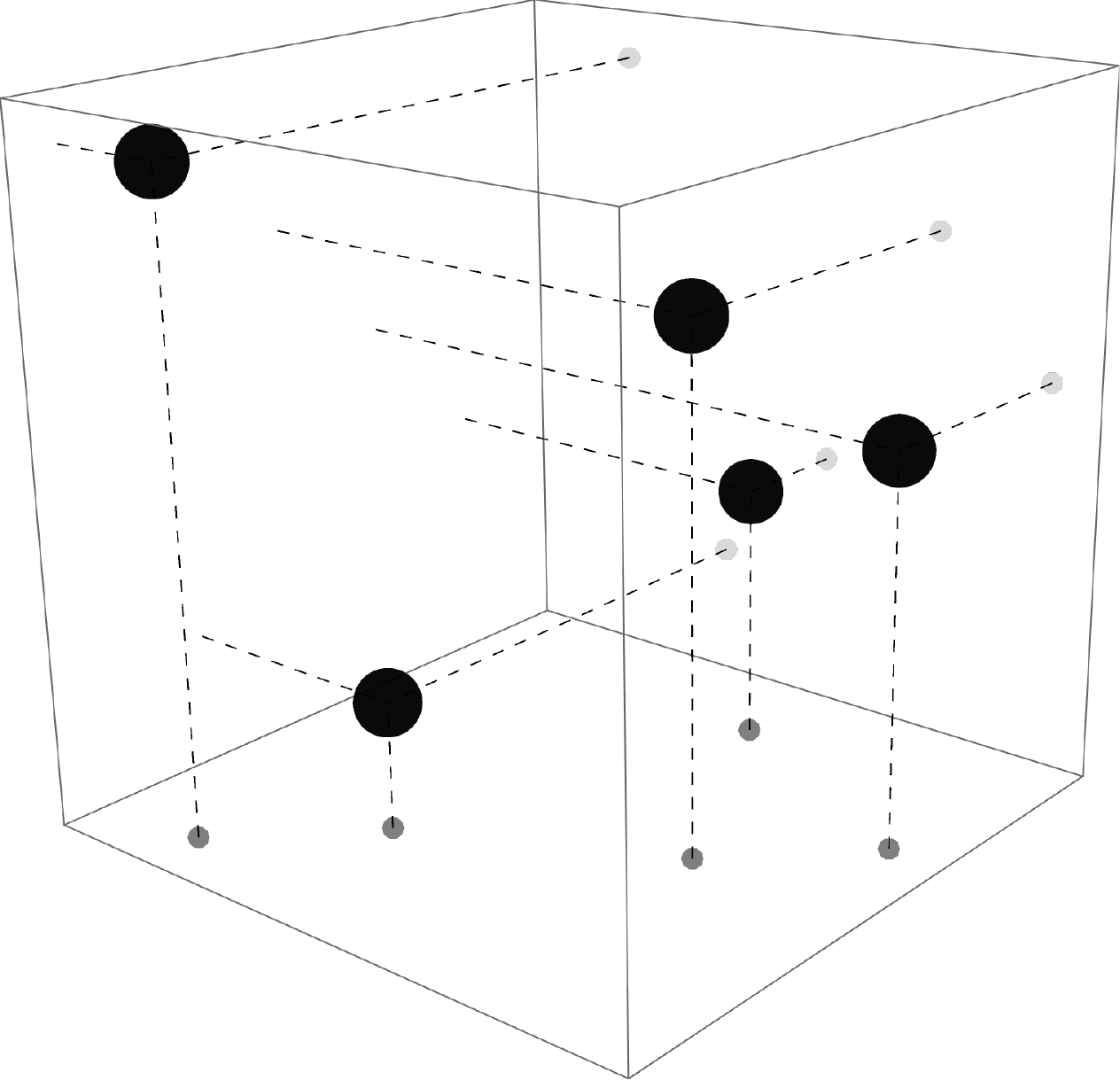}
\caption{The graph of $\Pi = (12534,51243)$}\label{fig:3dperm}
\end{center}
\end{figure}

{\bf Singleton $d$-dimensional mesh patterns.}  We define a \textit{$d$-dimensional singleton mesh pattern} ({\em $d$-SMP}) $P$ as a collection of distinct $d$-tuples that specify which $d$-hyperoctant should be forbidden (shaded).  These tuples can be coalesced as columns of a matrix which we denoted by $T(P)$.  Specifically a $d$-SMP of cardinality $k$, $P$, has an associated matrix
\renewcommand{\arraystretch}{1.25}
\[T(P)=
\left(\begin{array}{cccc}
p^1_1 & p^2_1 & \dots & p^k_1  \\
p^1_2 & p^2_2 & \dots & p^k_2  \\
\vdots &  & \dots & \vdots  \\
p^1_d & p^2_d & \ldots & p^k_d  \\
\end{array}\right)\]
where each $p_i^j\in\{+,-\}$.  The columns of $T(P)$ (as tuples) are the elements of the set $P$, so when considering $T(P)$ as a pattern, the order of columns is unimportant.  Each $d$-tuple in $P$ is responsible for specifying a $d$-hyperoctant which is to be shaded. For example, the 2-dimensional mesh pattern $\pattern{scale=0.5}{1}{1/1}{0/0,0/1,1/0}$ is of cardinality 3 and is defined by the 2-tuples $(-,+)$, $(-,-)$ and $(+,-)$ corresponding to the matrix
\renewcommand{\arraystretch}{1}
\[T(P)=\left(\begin{array}{rrr}
- & - & + \\
+ & - & - \\
\end{array}\right).\]  These tuples (columns of $T(P)$) correspond to shading the quadrants II, III and IV respectively. In short, $p^j_i$ indicates whether to move in the positive or negative direction with respect to the $i$-th coordinate.

\begin{definition}\label{def-occurrence} Given a $d$-dimensional permutation $\Pi=\Pi_1\Pi_2\dots \Pi_n$, we say that an element $\Pi_i$ of $\Pi$ is an {\em occurrence} of a $d$-SMP $P$ of cardinality $k$ if there does not exist an element $\Pi_j$ such that $\sgn(\Pi_j-\Pi_i)\in P$.  Here, $\sgn$ represents component-wise application of the usual sign function returning $+$ and $-$ instead of $1$ and $-1$ respectively.  Equivalently, $\Pi_i$ is an occurrence of $P$ if for any element $\Pi_j$, we have that $\sgn(\Pi_j-\Pi_i)\not \in P$. If $\Pi$ has no occurrences of $P$, we say that $\Pi$ {\em avoids} $P$. \end{definition}

Note that by this definition, a permutation of length $1$ is always an occurrence of any such pattern.  In Combinatorics on Words \cite{Loth}, a given set of prohibitions is {\em avoidable} if there exist arbitrarily long words avoiding it and it is {\em unavoidable} otherwise. The following definition introduces the relevant notions for  multidimensional permutations.

\begin{definition}\label{avoidability-def} A $d$-SMP $P$ is {\em avoidable} if there exist arbitrarily long $d$-dimensional permutations $\Pi$ that avoid $P$. If $P$ is not avoidable, it is {\em unavoidable}. Also, $P$ is {\em weakly avoidable} if there exists a $d$-dimensional permutation $\Pi$ of length $>1$ that avoids $P$. If $P$ occurs in every $d$-dimensional permutation then $P$ is {\em strongly unavoidable}. \end{definition}

\begin{remark} Clearly, strong unavoidability implies unavoidability, and avoidability implies weak avoidability.  \end{remark}

To illustrate Definition~\ref{avoidability-def}, note that in the 2-dimensional case, the pattern $\pattern{scale=0.5}{1}{1/1}{0/0,1/0}$ is strongly unavoidable as the minimal element of any permutation will be an occurrence, while the pattern  $\pattern{scale=0.5}{1}{1/1}{0/0,1/1}$ is avoidable as any permutation $12\ldots n$, for $n\geq 2$ avoids it.

Let us state some simple properties related to avoidability of multidimensional mesh patterns.

\begin{proposition}\label{inclusionproposition}
Suppose that $P$ and $P'$ are $d$-SMPs and $P \subseteq P'$.
\begin{itemize}
\item If $P$ is (weakly) avoidable, then $P'$ is (weakly) avoidable.
\item If $P'$ is (strongly) unavoidable, then $P$ is (strongly) unavoidable.
\end{itemize}
\end{proposition}

\begin{proof} The statements follow directly from the fact that if an element in a permutation is an occurrence of $P'$ then it is necessarily an occurrence of $P$, so that if a permutation avoids $P$ then it avoids $P'$ (the mesh pattern $P$ is more restrictive than $P'$ in the sense that there are more permutations avoiding $P'$ than $P$). \end{proof}

\begin{proposition}\label{automorhproposition}
Suppose that a $d$-SMP $P$ is avoidable (resp., unavoidable). Then every $d$-SMP $P'$ that is obtained from $P$ by one, or several of the following operations is avoidable (resp., unavoidable):
\begin{itemize}
\item a permutation  of the rows of $T(P)$;
\item complementing a row of $T(P)$, that is, replacing in the row each $+$ by $-$, and vice versa.
\end{itemize}
The same statements hold for weak avoidability and strong unavoidability.
\end{proposition}

\begin{proof} Suppose that $\Pi$ is a $P$-avoiding $d$-dimensional permutation and $P'$ is obtained from $P$ by a permutation $\tau\in S_d$ of the rows of $T(P)$.  Permuting the rows of $\Pi$ according to the permutation $\tau$ and then reordering the columns so that the first row is increasing (if necessary) yields $d$-dimensional permutation $\Pi'$ avoiding $P'$.

Also, if $P'$ is obtained from $P$ by changing each entry to the opposite in row $i \neq 1$ in $T(P)$ and $\Pi = (\pi^2,\pi^3, \ldots, \pi^d)$ is a $P$-avoiding permutation, then the permutation $\Pi'= (\pi^2,\ldots, \pi^{i-1},c(\pi^{i}),\pi^{i+1}, \ldots, \pi^d)$ obtained from $\Pi$ by taking the complement of $\pi^i$ is $P'$-avoiding. If $P'$ is obtained from $P$ by changing the sign of each entry in the first row in $T(P)$, then the permutation $\Pi'= (c(\pi^2),\ldots, c(\pi^d))$ is $P'$-avoiding.

Thus, in either case, $P$ is (weakly) avoidable/(strongly) unavoidable if and only if $P'$ is.
\end{proof}

To illustrate the operations in Proposition~\ref{automorhproposition}, consider the $5$-SMP $P$ defined by
\[T(P)=\left(\begin{array}{rrr}
+&-&+  \\
+&+&+  \\
-&-&+  \\
+&-&+  \\
+&+&-  \\
\end{array}\right)\]
Then, the patterns defined by \[\left(\begin{array}{rrr}
+&-&+  \\
+&+&-  \\
-&-&+  \\
+&-&+  \\
+&+&+  \\
\end{array}\right)\mbox{ and }\left(\begin{array}{rrr}
+&-&+  \\
-&-&-  \\
-&-&+  \\
+&-&+  \\
+&+&-  \\
\end{array}\right)\] are equivalent to $P$ in the sense that they have the same number of avoiders of each size. The first matrix is obtained from $T(P)$ by swapping rows 2 and 5 and the second is obtained from replacing the entries in row 2 to the opposite.

\section{Characterization of avoidable multidimensional singleton mesh patterns}\label{smp-characterization-sec}

Let $P$  be a  $d$-SMP. Define the \textit{rank} of the pattern $P$ to be the minimum cardinality of a pattern $P'$, $P' \subseteq P$, such that $T(P')$ has both a $+$ and $-$ in each row. If there is a row in $T(P)$ consisting entirely of $+$'s (or $-$'s), then we say that $\mbox{rank}(P) = \infty$. Otherwise, it is not hard to see that $2 \leq \mbox{rank}(P) \leq d$.  The lower bound is trivial and we can provide a construction that verifies the upper bound.  Suppose $d\geq 3$ and $P$ is a pattern of finite rank defined by the $k\geq 3$ $d$-tuples $P^1,P^2,\ldots,P^k$.  There must exist two tuples $P^i$ and $P^j$ such that they differ in sign in at least two positions.  Let $S=\{P^i,P^j\}$.  If $T(S)$ has both a $+$ and $-$ in each row, we are done.  Otherwise, let $c$ be index of the first row which does not contain both a $+$ and $-$.  Thus, as $P$ has finite rank, there must be a $P^r$ that could be added to $S$ to ensure that the $c$-th row of $T(S)$ has both a $+$ and $-$.  Continue this process until you arrive at a set $S$ such that $T(S)$ has both a $+$ and $-$ in each row;  by construction $|S|\leq d$.

Note that computing the rank of a $d$-SMP is an NP-hard problem.  In particular, it is a special case of the {\em set cover problem}, which is one of Karp's 21 NP-complete problems~\cite{Karp}. Suppose $\mathcal{S} = \{S^1, S^2,\ldots, S^k \}$ is a collection of subsets of the set $X = \{ 1, \ldots, d\}$. The set cover problem looks to identify the smallest sub-collection of $S$ whose union is $X$.

To illustrate the connection, consider $P=\{P^1,\dots, P^{k+1}\}$, a finite-rank pattern. Suppose we seek a $P'$ of minimal cardinality such that $T(P')$ has both a $+$ and $-$ in every row and suppose further, without loss of generality, that $P^{k+1}$ is an element of such a set $P'$ (going over all $k+1$ possibilities for including an element $P^i$  in $P'$ has no influence on the hardness of the problem).  For each $1\leq j< k+1$, we could identify $P^j$ with a subset $S_j=\{i| P^j_i\neq P^{k+1}_i\}$ and then let $S=\{S_j\}_{1\leq j\leq k}$.  It is then clear that finding a minimal $P'$ is equivalent to finding the smallest sub-collection of $S$ whose union is $X=\{1,\dots,d\}$.

The following theorem completely characterizes the avoidability of multidimensional mesh patterns.

\begin{theorem}\label{thm-main}
Let $P$ be a $d$-SMP.  If $\mbox{rank}(P) = \infty$ then $P$ is strongly unavoidable. If $\mbox{rank}(P) <\infty$ then $P$ is avoidable.
\end{theorem}

The theorem is a consequence of the following two lemmas.

\begin{lemma}
Let $P$ be a  $d$-SMP such that $\mbox{rank}(P) = \infty$. Then, every $d$-dimensional permutation $\Pi$ has an occurrence  of $P$, that is, $P$ is strongly unavoidable.
\end{lemma}

\begin{proof}
By Proposition~\ref{automorhproposition}, we can assume without loss of generality that every $d$-tuple in $P$ is of the form $(-, p_2^j, \ldots, p_d^j)$, $p_i^j \in \{ +, -\}$. Then, for every $d$-dimensional permutation $\Pi$, the element $\Pi_1$ is an occurrence of the pattern $P$ as the first component of $\sgn(\Pi_j-\Pi_1)$ is always + which ensures that $\sgn(\Pi_j-\Pi_1)\not \in P$ for any $j\neq 1$.
\end{proof}

Note that the following lemma establishes a stronger fact than avoidability; it shows that pattern-avoiding permutations exists of {\em each} length greater than or equal to $\mbox{rank}(P)<\infty$.

\begin{lemma}\label{avoid-lemma}
Let $P$ be a  $d$-SMP of rank $k<\infty$. If $n \geq k$, then there exists a $d$-dimensional permutation $\Pi$ of length $n$ that avoids $P$.
\end{lemma}

\begin{proof}
By Proposition~\ref{inclusionproposition}, we can assume that $d$-SMP $P$ has cardinality $k$. Suppose that $T(P)$ is the matrix
\[T(P)=
\left(\begin{array}{cccc}
p^1_1 & p^2_1 & \dots & p^k_1  \\
p^1_2 & p^2_2 & \dots & p^k_2  \\
\vdots &  & \dots & \vdots  \\
p^1_d & p^2_d & \ldots & p^k_d  \\
\end{array}\right)\]
where $p_i^j \in \{+,-\}$. Since $\mbox{rank}(P)<\infty$, every row of $T(P)$ contains both $+$ and $-$.

Given any permutation $\tau\in S_n$, we define the \textit{signature} of $\tau$ to be the tuple $s(\tau)=(s^1,\dots,s^n)$ where for $1\leq i \leq n$, $s^i=+$ if there is an ascent of $\tau$ at position $i$ and $s^i=-$ if there is a descent.  Similarly, $s^n$ is $+$ if $\tau_n<\tau_1$ and $-$ otherwise.  For any tuple $s=(s^1,\dots, s^n)$, let $S_n^s$ denote the set of $n$-permutations with signature $s$.  It is not hard to show that $|S_n^s|\neq 0$ if and only if $s$ contains both $+$ and $-$.

Suppose that $n = tk + r$, $t \geq 1$, $0 \leq r \leq k-1$.
Let $\tau_1, \tau_2,\ldots, \tau_d$ be a collection of  permutations, where $\tau_i \in S_n $ has signature
$$s^i = (p^1_i, \ldots, p^k_i,p^1_i, \ldots, p^k_i, \ldots, p^1_i, \ldots, p^k_i, p^1_i, \ldots, p^r_i).$$
Note that each $s^i$ still contains both $+$ and $-$. Consider the $d$-dimensional permutation $\Pi = (\tau_1^{-1} \tau_2, \tau_1^{-1} \tau_3,\ldots, \tau_1^{-1} \tau_d)$ of length $n$. We have that $\Pi$ avoids the pattern $P$, since the $i$-th element of $\Pi$, $1\leq i <n$, is not an occurrence of $P$ because of the $(i+1)$-st element, while the $n$-th element of $\Pi$ is not an occurrence of $P$ because of the first element.
\end{proof}

As an immediate corollary of Theorem~\ref{thm-main} we have the following result.

\begin{corollary}\label{complexity-avoidability} Let $P$ be a $d$-SMP of cardinality $k$. Then, recognizing avoidability of $P$ is an $O(d\cdot k)$ problem.
\end{corollary}

\begin{proof} By Theorem~\ref{thm-main}, $P$ is avoidable if and only if there is no row in $T(P)$ consisting entirely of $+$'s or $-$'s. To check this condition, we need $k-1$ comparisons of adjacent elements in each row that proves our claim since there are $d$ rows. \end{proof}

The following theorem complements Lemma \ref{avoid-lemma} by proving that {\em any} permutation of length less than $\mbox{rank}(P)<\infty$ necessarily has an occurrence of~$P$.

\begin{theorem}\label{small-lengths-thm}
Let $P$ be a  $d$-SMP of rank $k$, $2 \leq k <\infty$. Then every $d$-dimensional permutation of length $n$, $1\leq n < k$ has an occurrence of~$P$.
\end{theorem}

\begin{proof}
The statement is true for $n=1$ by definition. Let $n\geq 2$ and assume, in contrast, that some $d$-dimensional permutation $\Pi =(\pi^2,\pi^3, \ldots, \pi^d)$ of length $n$, $n < k$, avoids $P$, where $\pi^j=\pi_1^j\pi_2^j\ldots \pi_n^j$. From Definition~\ref{def-occurrence}, it follows that there is a collection of $\ell$, $2\leq \ell \leq n<k$, elements $\Pi_{i_1}, \ldots, \Pi_{i_{\ell}}$ of  $\Pi$
such that, for all $1\leq j\leq \ell-1$,  the tuples
\begin{equation}\label{set-1}\sgn(\Pi_{i_{j+1}}-\Pi_{i_{j}})\end{equation}
 and the tuple
\begin{equation}\label{set-2}\sgn(\Pi_{i_1}-\Pi_{i_\ell})\end{equation}
can be found as columns in $T(P)$ (with possible repetitions). That is, $\Pi_{i_1}$ is not an occurrence of $P$ because of $\Pi_{i_2}$,  $\Pi_{i_2}$ is not an occurrence of $P$ because of $\Pi_{i_3}$, etc,  $\Pi_{i_\ell}$ is not an occurrence of $P$ because of $\Pi_{i_1}$. Let $T'(P)$ be the set of all columns in $T(P)$ that are given by (\ref{set-1}) and (\ref{set-2}). Since $\ell<k=\mbox{rank}(P)$, $T'(P)\subset T(P)$ and by definition of rank, there is a row in $T'(P)$, say row $i$ whose elements are all of the same sign.  Without loss of generality, let's assume that the entire row consists of $+$'s. But then $$\pi_{i_1}^i<  \pi_{i_2}^i <\cdots<\pi_{i_\ell}^i<\pi_{i_1}^i,$$
which is a contradiction. Thus, $\Pi$ has an occurrence of $P$.
\end{proof}

\begin{definition} For $d$-SMPs $P_1$ and $P_2$, $P_1\bigvee P_2$ (resp., $P_1 \bigwedge P_2$) is  the $d$-SMP obtained by taking the union (resp., intersection) of  the columns of $T(P_1)$ and $T(P_2)$.\end{definition}

Table~\ref{avoid-union-intersection-tab} is a direct corollary of Theorem~\ref{thm-main}.  In this table A indicates avoidable, U indicates unavoidable and I indicates indeterminate.  To illustrate indeterminate, consider unavoidable $P_1=\pattern{scale=0.5}{1}{1/1}{1/1}$ and $P_2=\pattern{scale=0.5}{1}{1/1}{0/1}$ giving unavoidable $P_1\bigvee P_2=\pattern{scale=0.5}{1}{1/1}{1/1,0/1}$, while unavoidable $P_1=\pattern{scale=0.5}{1}{1/1}{1/0}$ and $P_2=\pattern{scale=0.5}{1}{1/1}{0/1}$ give avoidable $P_1\bigvee P_2=\pattern{scale=0.5}{1}{1/1}{1/0,0/1}$. For avoidable $P_1$ and $P_2$, $T(P_1 \bigwedge P_2)$ can be a single column or empty, so that  $P_1 \bigwedge P_2$ is unavoidable, while it is easy to construct an example of avoidable $P_1 \bigwedge P_2$ for avoidable $P_1$ and $P_2$. Also, for unavoidable $P_1$ and $P_2$, $T(P_1 \bigwedge P_2)$ is either empty, or it contains a row having the same sign.

Let Av$^d_n(P)$ be the set of $d$-dimensional permutations of length $n$ avoiding $P$ and $|$Av$^d_n(P)|$ is the cardinality of Av$^d_n(P)$. For $d$-SMPs $P_1$ and $P_2$ such that $P_1\subseteq P_2$, clearly, $|$Av$^d_n(P_1)|\leq |$Av$^d_n(P_2)|$. This observation immediately leads to the following results:
\begin{itemize}
\item  $|$Av$^d_n(P_1\bigvee P_2)|\geq \max\{|$Av$^d_n(P_1)|$,\ $|$Av$^d_n(P_2)|$\};
\item $|$Av$^d_n(P_1 \bigwedge P_2)|\leq \min\{|$Av$^d_n(P_1)|$,\ $|$Av$^d_n(P_2)|$\}.
\end{itemize}

\begin{table}
\begin{center}
\begin{tabular}{c|c|c|c}
$P_1$ & $P_2$ & $P_1 \bigvee P_2$ & $P_1 \bigwedge P_2$ \\
\hline
A & A & A & I \\
A & U & A & U \\
U & U & I & U \\
\end{tabular}
\caption{Avoidability/unavoidability of $P_1 \bigvee P_2$ and $P_1 \bigwedge P_2$ for $P_1\neq P_2$}\label{avoid-union-intersection-tab}
\end{center}
\end{table}

\begin{definition} A $d$-dimensional permutation $\Pi$ avoids simultaneously patterns in a set $S=\{P_1, P_2,\ldots,P_{\ell}\}$, $\ell\geq 2$,  if $\Pi$ avoids each pattern $P_i$. The set $S$  is {\em avoidable} if there exist arbitrarily long $d$-dimensional permutations $\Pi$ that avoid simultaneously patterns in $S$. If $S$ is not avoidable, it is {\em unavoidable}. Also, $S$ is {\em weakly avoidable} if there exists a $d$-dimensional permutation $\Pi$ of length $>1$ that avoids $S$. If every $d$-dimensional permutation contains an occurrence of a pattern in $S$ then $S$ is {\em strongly unavoidable}.\end{definition}

The following notion generalizes the notion of an inflation of an element in a 2-dimensional permutation used in the literature, namely, when an element is replaced by a permutation of consecutive elements.

\begin{definition}\label{inflation-def} Let  $\Pi = \Pi_1\Pi_2\dots \Pi_n = (\pi^2,\pi^3, \ldots, \pi^d)$  and $\Sigma = (\sigma^2,\sigma^3, \ldots, \sigma^d)$ be $d$-dimensional permutations of lengths $n$ and $m$ respectively. Then, the {\em inflation} of $\Pi_i$ by $\Sigma$ is the $d$-dimensional permutation of length $n+m-1$ obtained from the tuple of permutations $(\tau^2,\dots,\tau^d)$ where \[\tau^j=(\mu_1^j,\dots,\mu_{i-1}^j,\nu_1^j,\nu_2^j,\dots,\nu_m^j,\mu_{i+1}^j,\dots,\mu_n^j)\] with $\mu$ and $\nu$ defined as follows:
$\mu^j_1\dots  \mu^j_{i-1}\nu^j_1\dots  \nu^j_m  \mu^j_{i+1}\dots  \mu^j_{n}$ is a permutation of $\{1,2,\ldots,n+m-1\}$ such that $ \mu^j_s< \mu^j_t$ if and only if $\pi^j_s<\pi^j_t$ for $s,t\neq i$, $ \nu^j_s < \nu^j_t$ if and only if $\sigma^j_s<\sigma^j_t$, and $ \mu^j_s < \nu ^j_t$ if and only if $\pi^j_s<\pi^j_i$ for all $t$ and $s\neq i$.
\end{definition}

For example, the inflation of the second element of the permutation $(2413, 1243)$ by the permutation $(21,12)$ is the permutation $(25413,12354)$.

\begin{lemma}\label{lemma-inflation} Referring to the notation in Defenition~\ref{inflation-def}, if $\Pi$ (weakly) avoids a $d$-SMP $P_1$ and $\Sigma$ (weakly) avoids a $d$-SMP $P_2$ then the $d$-dimensional permutation $\Gamma$ obtained by inflation of {\em each} element of $\Pi$ by $\Sigma$ (weakly) avoids both $P_1$ and~$P_2$. \end{lemma}

\begin{proof} It is clear that $\Gamma$ avoids $P_2$ because each element $\Gamma_i$ in it  is part of a smaller permutation obtained from the $P_2$-avoiding $\Sigma$ by replacing elements in an order-isomorphic way (for $\Gamma_i$ there will be another element in $\Gamma$ in a shaded area given by $P_2$). On the other hand, no element $\Gamma_i$ can be an occurrence of $P_1$. Indeed, $\Gamma_i$ belongs to many sets of elements of $\Gamma$ that are placed in $\Gamma$ in an order-isomorphic to $\Pi$ way, and since $\Pi$ is $P_1$-avoiding, $\Gamma_i$ cannot be an occurrence of $P_1$.\end{proof}

\begin{theorem} A set of $d$-SMPs $S=\{P_1, P_2,\ldots,P_{\ell}\}$ is (strongly) unavoidable if there exists a $P_i$ that is (strongly) unavoidable. $S$ is (weakly) avoidable if each pattern $P_i$ is (weakly) avoidable. \end{theorem}

\begin{proof} The first statement is trivially true. As for the second statement, assume that a $d$-dimensional permutation $\Pi^{(i)}$ (weakly) avoids $P_i$ for $1\leq i\leq \ell$. Then, by Lemma~\ref{lemma-inflation}, inflation $\Sigma^{(2)}$ of {\em each} element in $\Pi^{(2)}$ by $\Pi^{(1)}$ (weakly) avoids both $P_1$ and $P_2$; inflation $\Sigma^{(3)}$ of {\em each} element in $\Pi^{(3)}$ by $\Sigma^{(2)}$ (weakly) avoids $P_1$, $P_2$, and $P_3$; and so on, until we obtain that inflation $\Sigma^{(\ell)}$ of {\em each} element in $\Pi^{(\ell)}$ by $\Sigma^{(\ell-1)}$ (weakly) avoids all patterns in $S$. Since $\Pi^{(i)}$'s can be arbitrary long, we see that $\Sigma^{(\ell)}$ can be arbitrary long showing that $S$ is (weakly) avoidable. \end{proof}

\section{Enumerative results for singleton mesh patterns}\label{enumeration-sec}

For a $d$-SMP $P$, the bivariate generating function
$$F_P(x,q):=\sum_{n\geq 0}x^n\sum_{\sigma\in S^d_n}q^{\# \mbox{{\small occurrences of}\ }P {\ in\ }\sigma}$$
gives the {\em distribution} of $P$. Let $F_d(x):=\sum_{n\geq 0}(n!)^{d-1}x^n = F_P(x,1)$ be the generating function of all permutations. For a formal power series $F(x)$, $[x^n]F(x)$ denotes the coefficient of $x^n$.

\subsection{Projective patterns}

\begin{definition}\label{projective-def} A $d$-SMP $P$ defined by the $d$-tuples
\begin{gather*}
(p^1_1, p^1_2, \ldots, p_{d}^1),  \\
\vdots \\
 (p_1^k, p_2^k, \ldots, p_{d}^k)
 \end{gather*}
is {\em projective} in direction $i$ if  $(p^j_1, \ldots, p^j_{i-1},+, p^j_{i+1}, \ldots, p^j_{d})$ is a column in $T(P)$ if and only if $(p^j_1, \ldots, p^j_{i-1},-, p^j_{i+1}, \ldots, p^j_{d})$ is also a column in $T(P)$. The $(d-1)$-SMP $P'$ where $T(P')$ is obtained by removing the $i$-th row of $T(P)$ is the {\em projection} of $P$ in direction $i$. For a projective $P$, we shorten $T(P)$ twice by placing a $\star$ in the $i$-th row. For projective patterns in several directions, we place a $\star$ in each projective direction.
\end{definition}

Up to symmetries, there are just two projective 2-SMPs, namely  $\pattern{scale=0.5}{1}{1/1}{0/0,0/1}$ and $\pattern{scale=0.5}{1}{1/1}{0/0,0/1,1/0,1/1}$. Regarding the former pattern, each 2-dimensional permutation of length $\geq 1$ contains exactly one occurrence of it, and thus $F_{\pattern{scale=0.5}{1}{1/1}{0/0,0/1}}(x,q)=1+q(F_2(x)-1)$. As for the latter pattern, only the permutation of length 1 contains an occurrence of it, so $$F_{\pattern{scale=0.5}{1}{1/1}{0/0,0/1,1/0,1/1}}(x,q)=qx+(F_2(x)-x)=(q-1)x+F_2(x).$$

By Theorem~\ref{thm-main}, in Definition~\ref{projective-def}, the projective pattern $P$  is avoidable (strongly unavoidable) if and only if the projection $P'$  is avoidable (strongly unavoidable). Moreover, if a $d$-dimensional permutation $\Pi$ of length $n$ contains $k$ occurrences of $P$, then the permutation $\Pi'$ obtained from $\Pi$ by removing the $i$-th row contains $k$ occurrences of $P'$ (if the first row is removed then we sort the columns of the obtained permutation (if needed) to make the new first row be increasing). In the opposite direction, if a $(d-1)$-dimensional permutation $\Pi'$ of length $n$ contains $k$ occurrences of $P'$ then, by inserting a new $i$-th row, there are $n!$ ways to extend $\Pi'$ to a $d$-dimensional permutation $\Pi$ of length $n$ with $k$ occurrences of $P$. The latter observations allow us to find distribution of all 3-dimensional projective patterns.

In what follows, Proposition~\ref{automorhproposition} allows us to assume that $P'$ is the projection of $P$ in direction $d$ (i.e.\ $i=d$) without loss of generality. For $d=3$, using the symmetries, we can assume that $P'$ in Definition~\ref{projective-def} is given by one of the following five patterns whose distribution is found below. The subscripts of the function $F$ correspond to the columns of $T(P)$.

\begin{center}
\begin{tabular}{ccccc}
$\pattern{scale=0.5}{1}{1/1}{1/1}$ & $\pattern{scale=0.5}{1}{1/1}{1/1,0/1}$ & $\pattern{scale=0.5}{1}{1/1}{0/1,1/0}$ & $\pattern{scale=0.5}{1}{1/1}{0/0,0/1,1/1}$ & $\pattern{scale=0.5}{1}{1/1}{0/0,0/1,1/0,1/1}$ \\
{\small Case 1} & {\small Case 2} & {\small Case 3} & {\small Case 4} & {\small Case 5} \end{tabular}
\end{center}

\noindent {\bf Case 1.} The distribution of the pattern $\pattern{scale=0.5}{1}{1/1}{1/1}$  on 2-dimensional permutations is the distribution of right-to-left maxima, which is the same as the distribution of cycles in permutations given by {\em signless Stirling numbers of the first kind}. It is not difficult to see, and can be found in \cite[Proposition 1.3.7]{Stanley}, that
$$\sum_{\sigma\in S_n}q^{\pattern{scale=0.5}{1}{1/1}{1/1}(\sigma)}=q(q+1)\cdots (q+n-1)=q^{(n)}$$
is the rising factorial, where$\pattern{scale=0.5}{1}{1/1}{1/1}(\sigma)$ is the number of occurrences of the pattern$\pattern{scale=0.5}{1}{1/1}{1/1}$ in $\sigma$. Hence, $F_{++\star}(x,q)=\sum_{n\geq 0}n!x^nq^{(n)}$. A straightforward generalization of this result is  $$F_{++\underbrace{\star\cdots\star}_{d-2 \mbox{ {\tiny times}}}}(x,q)=\sum_{n\geq 0}(n!)^{d-2}x^nq^{(n)}$$
where the entries of $++\underbrace{\star\cdots\star}_{d-2 \mbox{ {\tiny times}}}$ can be permuted without changing the distribution. \\[-2mm]

\noindent {\bf Case 2.} Each 3-dimensional permutation of length $\geq 1$ contains exactly one occurrence of the pattern, and thus $F_{+\star\star}(x,q)=1+q(F_3(x)-1)$. More generally, it is easy to see that
\begin{equation}\label{projective-hyperplane}F_{+\underbrace{\star\cdots\star}_{d-1 \mbox{ {\tiny times}}}}(x,q)=1+q(F_d(x)-1)\end{equation}
where the entries of $+\underbrace{\star\cdots\star}_{d-1 \mbox{ {\tiny times}}}$ can be permuted without changing the distribution. \\[-2mm]

\noindent {\bf Case 3.} According to~\cite{KitZha},
$$F_{\pattern{scale=0.5}{1}{1/1}{0/1,1/0}}(x,q)=\frac{F_2(x)}{1+x(1-q)F_2(x)}.$$
Hence, $$F_{+-\star,-+\star}=\sum_{n\geq 0}n!x^n[x^n]\frac{F_2(x)}{1+x(1-q)F_2(x)},$$
which can be easily generalized to $$F_{+-\underbrace{\star\cdots\star}_{d-2 \mbox{ {\tiny times}}},-+\underbrace{\star\cdots\star}_{d-2 \mbox{ {\tiny times}}}}=\sum_{n\geq 0}(n!)^{d-2}x^n[x^n]\frac{F_2(x)}{1+x(1-q)F_2(x)}$$ with the result being unchanged when the rows of the pattern are permuted. \\[-1mm]

\noindent {\bf Case 4.} It is easy to see that
$$F_{++\star,+-\star,--\star}(x,q)=1+\sum_{n\geq 1}n!(q(n-1)!+(n!-(n-1)!))x^n=$$
$$F_3(x)+(q-1)\sum_{n\geq 1}n!(n-1)!x^n.$$
It is straightforward to generalize this result to
$$F_{++\underbrace{\star\cdots\star}_{d-2 \mbox{ {\tiny times}}},+-\underbrace{\star\cdots\star}_{d-2 \mbox{ {\tiny times}}},--\underbrace{\star\cdots\star}_{d-2 \mbox{ {\tiny times}}}}(x,q)=F_d(x)+(q-1)\sum_{n\geq 1}(n!)^{d-2}(n-1)!x^n$$ where rows in the pattern can be permuted. \\[-2mm]

\noindent {\bf Case 5.} The distribution is clearly given by $qx + (F_3(x)-x)=(q-1)x+F_3(x)$, and more generally,
$$F_{\underbrace{\star\cdots\star}_{d \mbox{ {\tiny times}}}}(x,q)=(q-1)x + F_d(x).$$

\subsection{Antipodal patterns}

\begin{definition}\label{antipodal-def} Let $P$ be a singleton mesh pattern. For a column $C$ in $T(P)$, the complement $c(C)$ is obtained by replacing each $+$ by $-$ and each $-$ by $+$ in $C$.  The pattern $P$ is {\em plus-antipodal} (resp., {\em minus-antipodal}) if $C$ is a column in $T(P)$ if and only if $c(C)$ is (resp., not) a column in $T(P)$.\end{definition}

Examples of plus-antipodal patterns are $\pattern{scale=0.5}{1}{1/1}{0/1,1/0}$ and \begin{tiny}$\left( \begin{array}{rrrr} +&+&-&- \\ +&+&-&- \\ -&+&+&- \\ +&-&-&+\end{array} \right)$\end{tiny}. Examples of minus-antipodal patterns are $\pattern{scale=0.5}{1}{1/1}{0/0,0/1}$ and \begin{tiny}$\left( \begin{array}{rrrr} +&+&+&- \\ +&+&-&+ \\ +&-&+&+ \end{array} \right)$\end{tiny}. Note that for each minus-antipodal $d$-SMP $P$, $T(P)$ has $2^{d-1}$ columns, while for a plus-antipodal $d$-SMP $P$, $T(P)$ has an even number of columns between 0 and $2^d$.

\begin{theorem}\label{plus-antipodal-thm} Let $P$ be a plus-antipodal $d$-SMP with $T(P)$ having $2^{d} - 2$ columns and let $a_n$ denote the number of $d$-dimensional permutations of length $n$ avoiding $P$. Then,
$$A(x):=\sum_{n\geq 0}a_nx^n=\frac{F_d(x)}{1+F_d(x)};$$
$$F_P(x,q)=\frac{F_d(x)}{1+(1-xq)F_d(x)}.$$ \end{theorem}

\begin{proof}
Without loss of generality, assume that the columns that cannot be found in $T(P)$ are $+\cdots+$ and $-\cdots-$.  Each permutation either avoids $P$ or contains at least one occurrence of $P$. In the latter case, consider the lowest occurrence of $P$, that is, the element $a$ such that no other occurrence of $P$ has each coordinate smaller than the respective coordinate in $a$.  The occurrence $a$ gives the term $xq$, and we obtain the following functional equation, because with respect to $a$, in the region defined by $-\cdots-$ we must have a $P$-avoiding permutation (giving the term $A(x)$), while the region defined by $+\cdots+$ is independent from the rest of the permutation (giving the term $F(x,q)$; also note that there are no elements in any other region with respect to $a$ because of the element $a$).  Therefore, we have
$$F_P(x,q)=A(x)+xqA(x)F_P(x,q),$$
so that
\begin{equation}\label{eqn-F} F_P(x,q)=\frac{A(x)}{1-xqA(x)}.\end{equation}
To complete the proof, we derive the expression for $A(x)$ to be substituted in (\ref{eqn-F}). Note that
\begin{equation}\label{eq}a_{n+1}=((n+1)!)^{d-1}-\sum_{i=0}^{n}a_i((n-i)!)^{d-1}\end{equation}
where the first term is the number of all $d$-dimensional permutations of length $n+1$, and the second term is the number of all permutations containing at least one occurrence of $P$ (obtained by considering the lowest occurrence as in the arguments above). Multiplying both parts of (\ref{eq}) by $x^{n+1}$ and summing over all $n\geq 0$, we obtain
$$A(x)-1=F_d(x)-1-xA(x)F_d(x)$$
that leads to the desired result by solving for $A(x)$.\end{proof}

Our next result concerns plus-antipodal patterns of cardinality 2 and follows from a much stronger and more general Theorem~3.4 in \cite{GP}. It is interesting that the maximum number of occurrences of such patterns in $d$-dimensional permutations is equivalent to pattern-avoiding permutations in Theorem~3.4 in \cite{GP}.

\begin{theorem}\label{plus-antip-d-result}
The number of $d$-dimensional permutations $R(n)$ having $n$ occurrences of a plus-antipodal pattern $P$ with cardinality $2$ satisfies
$$\log R(n)  = \frac{(d-1)^2-1}{d-1}n \cdot \log n (1 + o(1))$$
as $n \rightarrow \infty$.
\end{theorem}

\begin{proof}  The symmetries described in Proposition \ref{automorhproposition} allow us to consider the pattern $P$ with $T(P) =$ \begin{tiny}$\left( \begin{array}{cc} +& - \\ +&-\\ \vdots & \vdots \\ +&- \end{array} \right)$\end{tiny}. It is clear that $\Pi = (\pi^2, \pi^3,\ldots,\pi^d)$ is a $d$-dimensional permutation of length $n$ with $n$ occurrences of $P$ if and only if there are no indices $i,j$ with $1\leq i < j \leq n$ such that $\pi^k_i<\pi^k_j$ for all $2\leq k\leq d$.  This can be phrased in the language of parallel pattern-avoidance (see Definition 2.5 and Definition 3.2 in \cite{GP}).  In their language, we have that $R(n)=S^{d-1}_n(12,\dots,12)$.  Applying Theorem 3.4 \cite{GP} then yields the result.
\end{proof}

With help of minus-antipodal patterns, we can characterize $d$-SMPs that have  no more than one occurrence in any permutation.

\begin{theorem}\label{one-occur-crit}
A $d$-SMP $P$ has no more than one occurrence in any $d$-dimensional permutation of length $n$ if and only if there is a minus-antipodal $d$-SMP $P'$ such that $P' \subseteq P$.
\end{theorem}

\begin{proof}
Assume that a pattern $P$ does not contain a minus-antipodal pattern. W.l.o.g. we can assume that columns $+ \cdots +$ and $- \cdots -$ do not belong to $T(P)$.  Then the permutation $\Pi = (\sigma^2, \sigma^3 \ldots, \sigma^d)$, where each $\sigma^i$ is the increasing permutation of length $12\dots n$, has $n$ occurrences of $P$.

On the other hand, assume that $P' \subseteq P$ for some  minus-antipodal $d$-SMP $P'$. Suppose that element $\Pi_i$ is an occurrence of $P$ in some permutation $\Pi$. Then for every $j \neq i$ an element $\Pi_j$ is not an occurrence of $P$ because of $\Pi_i$. So $\Pi$ has at most one occurrences of $P$.
\end{proof}

\begin{corollary}
Let $P$ be a $d$-SMP. If some $d$-dimensional permutation  has at least two occurrences of $P$, then for every $n$ there is a $d$-dimensional permutation of length $n$ with exactly $n$ occurrences of $P$.
\end{corollary}

\begin{proof}
The statement follows from the first paragraph in the proof of Theorem~\ref{one-occur-crit}.
\end{proof}

\subsection{Hyperplane patterns}

\begin{definition} A $d$-SMP $P$ is an {\em $i$-hyperplane $d$-SMP} if $${\underbrace{\star\cdots\star}_{i-1 \mbox{ {\tiny times}}}}+{\underbrace{\star\cdots\star}_{d-i \mbox{ {\tiny times}}}}\subseteq P,$$ that is, if $T(P)$ contains all possible columns with a $+$ in the $i$-th row. \end{definition}

Examples of 1-hyperplane patterns are $\pattern{scale=0.5}{1}{1/1}{0/1,1/1}$,  $\pattern{scale=0.5}{1}{1/1}{0/0,0/1,1/1}$ and \begin{tiny}$\left(\begin{array}{cccccc} +&+&+&+&-&- \\ +&+&-&-&+&- \\ +&-&+&-&+&- \end{array}\right)$\end{tiny}.

Since $\star \cdots \star + \star \cdots \star$ is a minus-antipodal pattern, Theorem~\ref{one-occur-crit} implies that any $d$-dimensional permutation $(\pi^2, \pi^3,\ldots, \pi^{d-1})$ has at most one occurrence of an $i$-hyperplane $d$-SMP $P$. 

Note that if $P={\underbrace{\star\cdots\star}_{i-1 \mbox{ {\tiny times}}}}+{\underbrace{\star\cdots\star}_{d-i \mbox{ {\tiny times}}}}$  then $P$ is a projective  pattern and its distribution, $1+q(F_d(x)-1)$, is given by (\ref{projective-hyperplane}) since a permutation of the rows of $T(P)$ does not change the distribution of the pattern (similar to the statement of Proposition~\ref{automorhproposition}). Thus, in what follows, we assume ${\underbrace{\star\cdots\star}_{i-1 \mbox{ {\tiny times}}}}+{\underbrace{\star\cdots\star}_{d-i \mbox{ {\tiny times}}}}\subset P$. The following theorem shows that finding the distribution of an $i$-hyperplane $d$-SMP can be reduced to finding the distribution of a $(d-1)$-SMP.

\begin{theorem}\label{hyperplane-reduction} Let $P = {\underbrace{\star\cdots\star}_{i-1 \mbox{ {\tiny times}}}}+{\underbrace{\star\cdots\star}_{d-i \mbox{ {\tiny times}}}} \bigvee B$, $B\neq \emptyset$, be an $i$-hyperplane $d$-SMP and $B^{(i)}$ is obtained from $B$ by removing the $i$-th entry, which is a minus, in each $d$-tuple. Also, assume that there are $f(n,k)$ $(d-1)$-dimensional permutations of length $n$ with $k$ occurrences of $B^{(i)}$. Then, there are $\sum_{k=1}^{n}k(n-1)!f(n,k)$ $d$-dimensional permutations of length $n$ with one occurrence of $P$, and the remaining $(n!)^{d-1}-\sum_{k=1}^{n}k(n-1)!f(n,k)$ permutations in $S^d_n$ avoid $P$.\end{theorem}

\begin{proof} The second claim follows from the first one and the observations that $|S^d_n|=(n!)^{d-1}$ and that $P$ occurs at most once in any permutation. We thus need to prove the first statement.

Let $\Pi=\Pi_1\Pi_2\dots \Pi_n=(\pi^2, \pi^3,\ldots, \pi^d)\in S^d_n$. For $2\leq i\leq d$, we let $$\Pi^{(i)}=\Pi'_1\Pi'_2\dots \Pi'_n:=(\pi^2,\ldots,\pi^{i-1},\pi^{i+1}, \pi^{i+2},\ldots, \pi^{d})\in S^{d-1}_n$$ and $\Pi^{(1)}=\Pi'_1\Pi'_2\dots \Pi'_n\in S^{d-1}_n$ is obtained from $\Pi$ by removing $\pi^1$ and replacing any other $\pi^i$ by the permutation $(\pi^2)^{-1}\pi^i$.

For $i\geq 2$, it is easy to see that  if $\Pi_j$ is an occurrence of $P$ in $\Pi$ then $\Pi_j'$ is an occurrence of $B^{(i)}$ in $\Pi^{(i)}$. Conversely, any  occurrence $\Pi'_j$ of $B^{(i)}$ in $\Pi^{(i)}$ can be ``lifted'' to the unique occurrence of $P$ in $\Pi$ by inserting a new $i$-th row (permutation) with the largest element being in column $j$. Note that this is the only possibility to create an occurrence of $P$ in $\Pi$ from an element in $\Pi'$ by inserting a new $i$-th row. Indeed, if $\Pi'_j$ is not an occurrence of $B^{(i)}$ because of an element $\Pi'_m$ (i.e.\ $\sgn(\Pi'_m-\Pi'_j)$ is a column in $T(B^{(i)})$) then $\Pi_j$ is not an occurrence of $P$ since $\sgn(\Pi_m-\Pi_j)$ is either a column in $T(B)$ (if the largest entry in a new row $i$ is in column $j$) or another column in $T(P)$. On the other hand, if $\Pi'_j$ is an occurrence of $B^{(i)}$ and the largest element in a new $i$-th row is not in column $j$ but in column $m$, $m\neq j$, then $\sgn(\Pi_m-\Pi_j)$ is a column in $T(P)\backslash T(B)$ so $\Pi_j$ is not an occurrence of $P$.

For $i=1$, again it is easy to see that  if $\Pi_j$ is an occurrence of $P$ in $\Pi$ then $\Pi_j'$ is an occurrence of $B^{(1)}$ in $\Pi^{(1)}$ as a permutation of columns does not affect anything. Conversely, any  occurrence $\Pi'_j$ of $B^{(1)}$ in $\Pi^{(1)}$ can be ``lifted'' to the unique occurrence of $P$ in $\Pi$ by inserting a new first row (permutation) with the largest element being in column $j$ and then by multiplying each row by $(\pi^1)^{-1}$ (to make the first row be the increasing permutation). A justification that this describes the unique way to create a permutation $\Pi$ with a single occurrence of $P$ by inserting the first row is similar to the case of $i\geq 2$ and hence is omitted.

In either case, for any permutation counted by $f(n,k)$, there are $k$ ways to choose an occurrence of $B^{(i)}$ to be made the only occurrence of $P$, and there are $(n-1)!$ ways to choose a permutation of length $n$ to insert (since the largest element $n$ must be in a specified position). The desired result is then obtained by summing over all possible $k\geq 1$.
\end{proof}

\section{Generalizations}\label{generalization-any-mesh-sec}

In the $2$-dimensional case, a permutation $\tau = \tau_1\tau_2 \dots \tau_{k}$ of length $k$ {\em occurs as a subpermutation} in a permutation $\sigma = \sigma_1\sigma_2 \dots \sigma_n$ of length $n$, $k\leq n$, if there exist $1\leq i_1<i_2< \cdots < i_{k}\leq n$ such that $\tau_{\ell}<\tau_{m}$ if and only if $\sigma_{i_\ell}<\sigma_{i_m}$, for $1\leq \ell<m\leq k$. Similarly, a $d$-dimensional permutation $\Psi = (\tau^2,\tau^3, \ldots, \tau^d)$ of length $k$ \textit{occurs as a subpermutation} in a $d$-dimensional permutation $\Pi = (\pi^2,\pi^3, \ldots, \pi^d)$ of length $n$, $k\leq n$, if there exist  $1\leq i_1<i_2< \cdots < i_{k}\leq n$ such that for each $j = 1,2, \ldots, d-1$, $\tau^j_{\ell}<\tau^j_{m}$ if and only if $\pi^j_{i_\ell}<\pi^j_{i_m}$, for $1\leq \ell<m\leq k$.

\subsection{General multidimensional mesh patterns}\label{general-mmp}

A {\em $d$-dimensional mesh pattern ($d$-MP) $\mathcal{P}$ of length $k$} is a pair $(T, P)$, where $T$ is a $d$-dimensional permutation of length $k$ and $P$ is a $d$-dimensional $(0,1)$-matrix of order $k+1$.  We denote by $supp(P)$  the \textit{support} of $P$, which is the set of all nonzero entries in the matrix $P$ defining the forbidden areas.

We say that a $d$-dimensional permutation $\Pi = \Pi_1\Pi_2 \dots \Pi_n$ defined by $(\pi^2,\pi^3, \ldots, \pi^d)$ \textit{contains an occurrence} of a mesh pattern  $\mathcal{P} = (T,P)$ of length $k$ if
\begin{itemize}
\item $T$ occurs in $\Pi$ as a subpermutation $\Pi_{i_1}\Pi_{i_2} \dots \Pi_{i_{k}}$  such that
\item there is no $r \in \{1, \ldots, n \}\backslash\{i_1,\ldots,i_k\}$ and no entry $(p_1,\ldots, p_d)\in supp(P)$ such that
\item $i_{p_1 -1} < r < i_{p_1}$ (where $i_0:=0$ and $i_{k+1}:=\infty$) and
\item either $\pi^j_{i_{p_{j+1} -1}}< \pi^j_r < \pi^j_{i_{p_{j+1}}}$ or $\pi^j_{i_{p_{j+1} -1}}> \pi^j_r > \pi^j_{i_{p_{j+1}}}$ for all $j = 1, \ldots, d-1$ (where the inequalities involving the non-defined $\pi^j_0$ or $\pi^j_{k+1}$ are assumed to be satisfied), that is, no element in $\Pi$ occurs in a forbidden area with respect to the subpermutation $\Pi_{i_1}\Pi_{i_2} \dots \Pi_{i_{k}}$.
\end{itemize}
If $\Pi$ has no occurrences of $\mathcal{P}$ then $\Pi$ {\em avoids} $\mathcal{P}$.

Our definition of a $d$-MP is consistent with the notion of a (2-dimensional) mesh pattern introduced in \cite{BrCl}. We next derive an enumerative result, to be referred to in Section~\ref{simplest-mmmp}, for the pattern $\mathcal{P}_d=((\underbrace{12,\ldots,12}_{d-1 \mbox{ {\tiny times}}}),\emptyset)$. Let $F_{n,d}(q)$ be the generating function for the distribution of $\mathcal{P}_d$ on $S^d_n$, the set of $d$-dimensional permutations of length $n$. Clearly,
$F_{1,d}(q)=1$ and $F_{2,d}(q)=2^{d-1}-1+q$ (as all permutations but $(\underbrace{12,\ldots,12}_{d-1 \mbox{ {\tiny times}}})$ avoid $\mathcal{P}_d$). Moreover, $F_{3,2}(q)=q^3+2q^2+2q+1$, where the coefficient of $q^3$ is given by the 2-dimensional permutation $(123)$, the coefficient of $q^2$ is given by $(132)$ and $(213)$, the coefficient of $q$ is given by $(231)$ and $(312)$, and the coefficient of $q^0$ is given by $(321)$. One can also compute
$$F_{3,3}(q)=q^3+6q^2+12q+17\ \mbox{ and \ } F_{3,4}(q)=q^3+14q^2+50q+151.$$
The following result generalizes the last three formulas.

\begin{theorem}\label{d-mp-121212} For the pattern $\mathcal{P}_d$, $d\geq 2$,
$$F_{3,d}(q)=q^3+2(2^{d-1}-1)q^2+(3^d-2^{d+1}+1)q+(6^{d-1}-3^d+2^d).$$\end{theorem}

\begin{proof} Clearly, any permutation in $S^d_3$ has at most three occurrences of $\mathcal{P}_d$ and the only $d$-dimensional permutation with three occurrences is $(\underbrace{123,\ldots,123}_{d-1 \mbox{ {\tiny times}}})$.

For convenience, we will denote elements of a permutation $\Pi\in S^d_3$ by $a$, $b$, and $c$.
To have two occurrences of $\mathcal{P}_d$ in a permutation $abc$ given by $(\pi^2,\pi^3, \ldots, \pi^d)\in S^d_3$, either
\begin{itemize}
\item each $\pi^i\in \{123,132\}$ and there exists $\pi^j=132$, for $1\leq j\leq d-1$ ($ab$ and $ac$ are occurrences, $bc$ is not an occurrences because of $\pi^j$), or
\item each $\pi^i\in \{123,213\}$ and there exists $\pi^j=213$, for $1\leq j\leq d-1$ ($ac$ and $bc$ are occurrences, $ab$ is not an occurrences because of $\pi^j$).
\end{itemize}
As the cases are not overlapping, and in each of them we have $2^{d-1}-1$ permutations, we get the desired coefficient of $q^2$. Similarly, to have exactly one occurrence of $\mathcal{P}_d$, either
\begin{itemize}
\item[(a)] each $\pi^i\in \{123,132,231\}$ and there exists $\pi^j=231$, for $1\leq j\leq d-1$ ($ab$ is the occurrence; there are $3^{d-1}-2^{d-1}$ possibilities here), or
\item[(b)] each $\pi^i\in \{123,213,312\}$ and there exists $\pi^j=312$, for $1\leq j\leq d-1$  ($bc$ is the occurrence; there are $3^{d-1}-2^{d-1}$ possibilities here), or
\item[(c)] each $\pi^i\in \{123,132,213\}$ and there exists $\pi^j=132$ and $\pi^m=213$, for $1\leq j,m\leq d-1$  ($ac$ is the occurrence; using the inclusion-exclusion principle, there are $3^{d-1}-2\cdot 2^{d-1}+1$ possibilities here).
\end{itemize}
Since the three cases are not overlapping, we obtain the desired coefficient of $q$. The coefficient of $q^0$ is obtained by subtracting the other coefficients from the total number of permutations $6^d$.
\end{proof}

We note that the coefficient of $q$ in Theorem~\ref{d-mp-121212} appears as the sequence A028243 in \cite{oeis} ($2, 12, 50, 180, 602, 1932, 6050,\ldots$), which is doubled Stirling numbers of the second kind, given by the formula $2S(n,3)$ and has several interesting combinatorial interpretations.  We can explain combinatorially, for example, the fact that permutations in $S^d_3$ with one occurrence of the pattern $\mathcal{P}_d$ are in bijection with strings over the alphabet $\{0,1,2\}$ of length $d$ that contain at least one 0 and one 1. For example, for $d=2$ such strings are $01$ and $10$, and for $d=3$ such strings are the three permutations of 100, the three permutations of 110, and the three permutations of 210. Referring to the respective cases in the proof of Theorem~\ref{mmp-pluses-1}, a bijection can be described as follows. We map a string $s_1s_2\dots s_d$ in question to a permutation $(\pi^2,\pi^3,\ldots,\pi^d)$ so that
\begin{itemize}
\item[(a)] if $s_1=0$ then $s_i\mapsto \pi^i$, $2\leq i\leq d$, as $0 \mapsto 132$, $1 \mapsto 231$, and $2 \mapsto 123$ thus giving a permutation in (a) (note at least one appearance of 231);
\item[(b)] if $s_1=1$ then $s_i\mapsto \pi^i$, $2\leq i\leq d$, as $0 \mapsto 312$, $1 \mapsto 213$, and $2 \mapsto 123$ thus giving a permutation in (b) (note at least one appearance of 312);
\item[(c)] if $s_1=2$ then $s_i\mapsto \pi^i$, $2\leq i\leq d$, as $0 \mapsto 132$, $1 \mapsto 213$, and $2 \mapsto 123$ thus giving a permutation in (c) (note appearances of at least one 132 and at least one 213).
\end{itemize}
The map described above is clearly a bijection.

\subsection{Multidimensional marked mesh patterns}\label{simplest-mmmp}

In the 2-dimensional case, {\em marked mesh patterns} ({\em MMPs}) are defined similarly to mesh patterns, but now each region (given by $P$) can be either shaded or it contains a non-negative integer. If a region in an MMP has an integer $t$, then in an occurrence of this MMP we require the respective region to have at least $t$ elements. The simplest marked mesh patterns of length 1 are known as {\em quadrant marked mesh patterns} ({\em QMMPs}) and they have been studies in several papers, e.g.\ in \cite{KitRem2012,KitRem2012-2,KitRemTie2012,QiuRem2017}.

We generalize the notion of a QMMP by modifying the definition of a $d$-SMP. A {\em $d$-dimensional simplest marked mesh pattern} ({\em $d$-SMMP}) $P$ of cardinality $k$ is a collection of $k$ $(d+1)$-tuples
\begin{gather*}
(p^1_1, p^1_2, \ldots, p_{d}^1,x^1),  \\
(p^2_1, p^2_2, \ldots, p_{d}^2,x^2),  \\
\vdots \\
 (p_1^k, p_2^k, \ldots, p_{d}^k,x^k),
  \end{gather*}
where $p_i^j \in \{+,-\}$,  $(p^i_1, \ldots, p_{d}^i)\neq (p^j_1,\ldots, p_{d}^j)$ for $i\neq j$, and $x^j$ is a positive integer or a $\blacksquare$. We think of the collection as a table $T=T(P)$ whose columns are the listed tuples.

\begin{definition}\label{def-occurrence-SMMP} An element $\Pi_i$ in a $d$-dimensional permutation $\Pi$,  is an {\em occurrence} of  a $d$-SMMP $P$ of cardinality $k$ if
\begin{itemize}
\item for any element $\Pi_j$,  we have that
$$(\sgn(j-i),\sgn(\pi^1_j-\pi^1_i), \sgn(\pi^2_j-\pi^2_i),\ldots, \sgn(\pi^{d-1}_j-\pi^{d-1}_i),\blacksquare)$$
is {\em not} a column in $T(P)$ (that is, no element is in the shaded region), and
\item if $(p^s_1, p^s_2, \ldots, p_{d}^s,x)$ is a column in $T(P)$ then there are at least $x$ elements $\Pi_j$ such that $\sgn(j-i)=p^s_1$, $\sgn(\pi^1_j-\pi^1_i)= p^s_2$, $\sgn(\pi^2_j-\pi^2_i)= p^s_3$, etc, $\sgn(\pi^{d-1}_j-\pi^{d-1}_i)= p^s_d$.
\end{itemize}
If $\Pi$ has no occurrences of $P$, we say that $\Pi$ {\em avoids} $P$. \end{definition}

Merging the approaches in Section~\ref{general-mmp} and Definition~\ref{def-occurrence-SMMP}, one can introduce the notion of a  (general) $d$-dimensional marked mesh pattern ({\em $d$-MMP}) where each region is required either to be empty or to contain at least $t\geq 0$ elements (the case of $t=0$ corresponds to having no requirements for such a region). However, due to space concern, we omit a formal definition of a $d$-MMP, instead stating an enumerative avoidance result for a $d$-SMMP $P$ with the single column in $T(P)$ being $(+,\ldots,+,1)$.

\begin{theorem}\label{mmp-pluses-1} The number of permutations in $S^d_3$ avoiding the  $d$-SMMP $P$ defined by $(\underbrace{+,\ldots,+,}_{d \mbox{ {\tiny times}}}1)$ is given by $6^{d-1}-3^d+2^d$.\end{theorem}

\begin{proof} We observe that a permutation in $S^d_3$ avoids $P$ if and only if it avoids the pattern $\mathcal{P}_d$ in Section~\ref{general-mmp}, so that the desired quantity is given by the coefficient of $q^0$ in $F_{3,d}(q)$ in Theorem~\ref{d-mp-121212}. \end{proof}

%
%

\section{Directions of further research}\label{research-directions-sec}

A $d$-dimensional permutation of length $n$ can contain $k$, $0\leq k\leq n$, occurrences of a $d$-SMP $P$. The extreme cases of $k=0$ (avoidance) and $k=n$, whenever they are feasible, are particularly interesting here. While the avoidance is a classical direction of research in the theory of permutation patterns, the other extreme case is rather specific to the patterns in question, and it ought to bring us to some interesting (enumerative or structural) results. A starting point could be understanding permutations of length $n$ having $n$ occurrences of the pattern $+\cdots+$ in $\geq 3$ dimensions. More generally, finding the distribution of the pattern $+\cdots+$, that would generalize the known distribution result for the pattern $\pattern{scale=0.5}{1}{1/1}{1/1}$ in two dimensions (corresponding to the right-to-left maxima in permutations and given by the signless Stirling numbers of the first kind \cite[Proposition 1.3.4]{Stanley}), is a good open challenging problem.

In Definition~\ref{antipodal-def}, we introduce the notion of a minus-antipodal SMP, but apart from the pattern $\underbrace{\star\cdots\star}_{i-1 \mbox{ {\tiny times}}}+\underbrace{\star\cdots\star}_{d-i \mbox{ {\tiny times}}}$  (that is also projective and hyperplane) and Theorem~\ref{one-occur-crit}, we do not provide any results for minus-antipodal patterns, while it seems to be an interesting and natural class of patterns. A similar situation is with another natural class, the class of plus-antipodal patterns (introduced in Definition~\ref{antipodal-def}) as essentially the only result we give for such patterns are those in Theorems~\ref{plus-antipodal-thm} and~\ref{plus-antip-d-result}.

Generalizing Theorem~\ref{mmp-pluses-1} to finding the distribution of the $d$-SMMP $P$ defined by $(\underbrace{+,\ldots,+,}_{d \mbox{ {\tiny times}}}1)$, or more generally, of the $d$-SMMP $P$ defined by $(\underbrace{+,\ldots,+,}_{d \mbox{ {\tiny times}}}x)$ for $x\geq 1$, would give an interesting generalization of the respective results in \cite{KitRem2012} for quadrant marked mesh patterns. We note that the arguments in \cite{KitRem2012} cannot be extended in a straightforward way to 3 or more dimensions.

Finally, a natural step is initiating (systematic) studies of $d$-dimensional mesh patterns of length 2, $d\geq 3$, that would extend the systematic studies in \cite{Hilmarsson2015Wilf} and \cite{KitZha} to higher dimensions. Also, various general equivalences of 2-dimensional mesh patterns \cite{Hilmarsson2015Wilf,Tenner} can be considered to be extended to higher dimensions. We note that the question on avoidability of a mesh pattern of length 2 or more is uninteresting (unlike the length 1 case) as at least one of the two monotone permutations (each column of which is the same monotone permutation, $12\dots n$ or $n(n-1)\dots 1$) will always avoid any such pattern.

\section*{Acknowledgements}
The work of Sergey Avgustinovich, Vladimir Potapov, and Anna Taranenko was carried out within the framework of the state contract of the Sobolev Institute of Mathematics (project no. FWNF-2022-0017).


\begin{thebibliography}{20}
\bibitem{Andre1} D. {Andr\'{e}}. D\'{e}veloppements de sec x et de tang x, \emph{C.
  R. Acad. Sci. Paris}, \textbf{88} (1879), 965--967.
\bibitem{Andre2} D. {Andr\'{e}}. M\'{e}moire sur les permutations altern\'{e}es,
{\em J. Math. Pur. Appl.}, \textbf{7} (1881), 167--184.
\bibitem{AM2010} A. Asinowski and T. Mansour. Separable $d$-permutations and guillotine partitions, {\em Ann. Comb.} {\bf 14} (2010), 17--43.
\bibitem{BrCl} P. Br\"and\'en and A. Claesson. Mesh patterns and the expansion of permutation statistics as sums of permutation patterns, {\em Elect. J. Comb.} {\bf 18(2)} (2011), \#P5, 14pp.
\bibitem{GP} B. Gunby and D. P\'av\"{o}lgyi. Asymptotics of pattern avoidance in the Klazar set partition and permutation-tuple settings, {\em Europ. J. Combin.} {\bf 82} (2019) 102992.
\bibitem{Hilmarsson2015Wilf}
I.~Hilmarsson, I.~J\'onsd\'ottir, S.~Sigurdard\'ottir, L. Vidarsd\'ottir, and H.~Ulfarsson. Wilf-classification of mesh patterns of short length,
  {\em Electr. J. Combin.} {\bf 22(4)} (2015),  \#P4.13.
\bibitem{Karp} R. M. Karp. Reducibility among combinatorial problems, in {\it Complexity of computer computations (Proc. Sympos., IBM Thomas J. Watson Res. Center, Yorktown Heights, N.Y., 1972)}, 85--103, Plenum, New York.
\bibitem{Kit2011} S. Kitaev. Patterns in Permutations and Words, Springer, 2011.
\bibitem{KR2007} S. Kitaev and J. Robbins. On multi-dimensional patterns, {\em Pure Math. and
Appl.} ({\em Pu.M.A.}) {\bf 18} (2007) 3--4, 1--9.
\bibitem{KitRem2012} S. Kitaev and J. Remmel. Quadrant marked mesh patterns, {\em J. Integer Sequences} {\bf 12}
(2012), Article 12.4.7, 29 pp.
\bibitem{KitRem2012-2} S. Kitaev and J. Remmel. Quadrant marked mesh patterns in alternating permutations, {\em S\'em. Lothar. Combin.} {\bf 68} (2012), Art. B68a, 20 pp.
\bibitem{KitRemTie2012} S. Kitaev, J. Remmel and M. Tiefenbruck. Quadrant marked mesh patterns in 132-avoiding permutations, {\em Pure Math. and
Appl.} ({\em Pu.M.A.}) {\bf 23} (2012) 3, 219--256.
\bibitem{KitZha} S. Kitaev, P. B. Zhang. Distributions of mesh patterns of short lengths, {\em Adv. in Appl. Math.} {\bf 110}
(2019), 1--32.
\bibitem{Loth} M. Lothaire. {\em Combinatorics on words}, Encyclopedia of Mathematics and its Applications {\bf 17}, Addison-Wesley Publishing Co., Reading, Mass, 1983.
\bibitem{QiuRem2017} D. Qiu and J. Remmel. Quadrant marked mesh patterns in 123-avoiding permutations, {\em Discrete Math. Theor. Comput. Sci.} {\bf 19} (2017) 2, Paper No.\ 12, 49 pp.
\bibitem{oeis} N. J. A. Sloane. The On-Line Encyclopedia of Integer Sequences, available at http://oeis.org.
\bibitem{Stanley} R. P. Stanley. {\em Enumerative Combinatorics}, Volumes 1. Cambridge University Press, 1997.
\bibitem{Tenner}  B. E. Tenner. Mesh patterns with superfluous mesh, {\em Adv. Appl. Math.}, {\bf 51} (2013), 606--618.
\bibitem{ZG2007} H. Zhang and D. Gildea. Enumeration of Factorizable Multi-Dimensional Permutations, {\em J. Integer Seq.} {\bf 10} (2007), Art.\ 07.5.8, 18pp.
\end{thebibliography}
\end{document}